\documentclass[12pt]{article}
\usepackage[usenames]{color}
\usepackage[colorlinks=true,
linkcolor=webgreen,
filecolor=webbrown,
citecolor=webgreen]{hyperref}

\definecolor{webgreen}{rgb}{0,.5,0}
\definecolor{webbrown}{rgb}{.6,0,0}
\usepackage{fullpage}
\usepackage{graphicx}
\usepackage{amscd}
\usepackage{color}
\usepackage{float}
\usepackage[utf8]{inputenc}
\usepackage[T1]{fontenc}
\usepackage{lmodern}
\usepackage{courier}
\usepackage{graphics,amsmath,amssymb}
\usepackage{amsthm}
\usepackage{amsfonts}
\usepackage{latexsym}
\usepackage{epsf}

\newtheorem{theorem}{Theorem}
\newtheorem{proposition}{Proposition}

\newtheorem{lemma}{Lemmma}
\newtheorem{conjecture}{Conjecture}

\begin{document}
\begin{center}
{\LARGE\bf An Integral Representation  of Kekul\'{e} \\
 Numbers, and Double Integrals Related \\
 to  Smarandache Sequences }

\vspace{.25in}
\large 
John M. Campbell \\
Toronto, ON \\
\href{mailto:jmaxwellcampbell@gmail.com}{\tt jmaxwellcampbell@gmail.com}
\end{center}

\begin{abstract}

  We present an integral representation of Kekul\'{e} numbers for $P_{2} (n)$ benzenoids.
 Related integrals of the form $\int_{-\pi}^{\pi} \frac{\cos(nx)}{\sin^{2}x +k } dx$ are evaluated.
 Conjectures relating double integrals of the form $\int_{0}^{m}  \int_{-\pi}^{\pi} \frac{\cos (2nx)}{k+\sin^{2}x} dx dk $ to
 Smarandache sequences are presented.

\end{abstract}

\section{Introduction}

  It is well-known that some characteristics of Kekul\'{e} structures may be enumerated in terms of recursive
 sequences \cite{KSBH} \cite{ITBH} \cite{hosoya}  \cite{ATBH} \cite{A122652}. A variety of discrete formulas for $K(B)$
 (the number of Kekul\'{e} structures in a benzenoid system $B$) are known \cite{KSBH}.

  The existence of Kekul\'{e} structures in benzenoid systems is commonly considered to be an algebraic problem.
 A well-known discrete identity follows: letting $A$ denote the adjacency matrix of a benzenoid
 system $B$ with $n$ vertices, 
 $\det A = (-1)^{\frac{n}{2}} K(B)^{2}$ (so a benzenoid system has Kekul\'{e} structures iff the matrix $A$ is
 non-singular) \cite{EKSBS} \cite{ATBH}, a result discovered by Dewar and Longuet-Higgins (1952) and 
 rigorously proven in 1980 by Cvetkovi\'{c} (et al.) \cite{KSBH}.

 Inspired by Fourier analysis, 
we prove an integral equation for Kekul\'{e} numbers for pyrenes on a string with hexagons removed from both sides ($K(P_{2} (n))$), and
  prove an explicit closed form evaluation of $\int_{-\pi}^{\pi} \frac{\cos(nx)}{\sin^{2}x +k } dx$.

  We conjecture how the related expressions $\int_{0}^{m}  \int_{-\pi}^{\pi} \frac{\cos (2nx)}{k+\sin^{2}x} dx dk $ may be used to represent (quadratic) 
continued fractions, and conjecture that double integrals of this form 
correspond to sequences such as the square root of smallest square of the type $k^{2}(k+1)$ and 
the number of solutions to $x^{2} \equiv 0 \mod n$,
 establishing unexpected number theoretic implications. Many of the sequences discussed below may be 
described as  Smarandache sequences, such as core($n$).

\section{Main Problems and Results}

  Given the tetravalency of carbon,  it becomes evident that the number of Kekul\'{e} 
structures of a benzenoid hydrocarbon is equal to the number of linear factors of the benzenoid system. Given the work of
Ohkami and Hosoya, the following recurrence formula for the Kekul\'{e} numbers for pyrenes on a string (i.e. a system of linearly annelated pyrenes, pyrenes
being a characteristic subunit of all-benzenoid systems) of the form $K(P_{2} (n))$ 
holds \cite{KSBH} \cite{A122652}:

\begin{equation}\label{pyrene}
K(P_{n})=10K(P_{n-1})-K(P_{n-2})
\end{equation}

  Consider the companion recursion sequence of \eqref{pyrene} given by 
$d_{n} = 10d_{n-1} - d_{n-2}$; $d_{0} = 1$, $d_{1} = 5$ \cite{A001079}. 
Surprisingly, these recursion sequences may be used to evaluate the following integral, letting $c_{n}=K(P_{n})$:

\begin{lemma}\label{lemmausedinfirsttheorem}

\begin{equation*}
\int_{-\pi}^{\pi} \frac{\cos (nx)}{2+\sin^{2}x} dx = \frac{(-1)^{n}+1}{2}  \pi (-c_{\lfloor \frac{n}{2}  \rfloor} 
 + d_{\lfloor \frac{n}{2}  \rfloor} \sqrt{\frac{2}{3}})
\end{equation*}

\end{lemma}

\begin{proof}

  First consider the non-zero values of the above equation; use induction.
 Let $P(n)$ be the statement ``$\int_{- \pi}^{\pi} \frac{\cos (2nx)}{2+\sin ^{2} x } dx =  (-c_{n} + d_{n} \sqrt{\frac{2}{3}}) \pi$''.
  Suppose $P(n)$ holds. So we want $P(n-2) \wedge P(n-1)  \Rightarrow P(n)$. Verify that $P(n)$ holds for all $n < m$ for
 some $2<m$.

  We have $K(P_{n})=10K(P_{n-1})-K(P_{n-2})$ \eqref{pyrene} and $d_{n} = 10d_{n-1} - d_{n-2}$; $d_{0} = 1$, $d_{1} = 5$ \cite{A001079}.

  So by the assumption of $P(n)$ we have:

\begin{equation*}
\int_{-\pi}^{\pi} \frac{10\cos (2(n-1)x)}{2+\sin^{2}x} dx =  10 (-c_{n-1}  + d_{n-1} \sqrt{\frac{2}{3}}) \pi
\end{equation*}

\begin{equation*}
\int_{-\pi}^{\pi} \frac{\cos (2(n-2)x)}{2+\sin^{2}x} dx =  (-c_{n-2}  + d_{n-2} \sqrt{\frac{2}{3}}) \pi
\end{equation*}

So it remains to prove that:

\begin{equation}\label{above2}
\int_{-\pi}^{\pi} \frac{10\cos (2(n-1)x)-\cos (2(n-2)x)-\cos (2nx)}{2+\sin^{2}x} dx =  0
\end{equation}

  It is elementary to evalute the above integral as follows:

\begin{equation*}
\int_{-\pi}^{\pi} \frac{10\cos (2(n-1)x)-\cos (2(n-2)x)-\cos (2nx)}{2+\sin^{2}x} dx =  \frac{4 \sin (2n \pi)}{n-1}
\end{equation*}

  So \eqref{above2} holds, letting $n$ be an integer (greater than 1).

For odd values of $n$ for $\int_{-\pi}^{\pi} \frac{\cos (nx)}{2+\sin^{2}x} dx$, interpret the integrand as a shifted skew-symmetric function.
 So it remains to prove that
 $ \frac{\cos (n(x+\frac{\pi}{2}))}{2+\sin^{2}(x+\frac{\pi}{2})} = -\frac{\cos (n(-x+\frac{\pi}{2}))}{2+\sin^{2}(-x+\frac{\pi}{2})}$, which is elementary. 

\end{proof}

\begin{theorem}\label{main}

\begin{equation*}
K(P_{n}) = -\frac{\int_{-\pi}^{\pi}   \frac{\cos(2nx)}{2+\sin^{2}x} dx}{\pi}+(-1)^{n} \cos(2n \sin^{-1} \sqrt{3}) \sqrt{\frac{2}{3}}
\end{equation*}

\end{theorem}

\begin{proof}

  An alternative evaluation of $d_n$, from  Jasinski (2008) \cite{A001079} ,  follows:

\begin{equation*}
d_{n}=(-1)^{n} \cos(2n \sin ^{-1} \sqrt{3})
\end{equation*}

So given Lemma~\ref{lemmausedinfirsttheorem} , Theorem~\ref{main} holds.

\end{proof}

\begin{lemma}\label{lemma2}

 Let $y_{n}$ be defined by the following integer sequence: $y_{n} = (4k+2) y_{n-1} - y_{n-2}$, $y_{0}  = 0$, $y_{1}=4$ for some fixed $k$.
 Let $z_{n}$ be defined by the following integer sequence: $z_{n} = (4k+2) z_{n-1} - z_{n-2}$, $z_{0}  = 1$, $z_{1}=2k+1$ for some fixed $k$.
  Then:

\begin{equation}
\int_{-\pi}^{\pi} \frac{\cos (nx)}{k+\sin^{2}x} dx = \frac{(-1)^{n}+1}{2}  \pi (-y_{\lfloor \frac{n}{2}  \rfloor}
  + z_{\lfloor \frac{n}{2}  \rfloor} \sqrt{\frac{2}{T_{k}}})
\end{equation}

\end{lemma}

\begin{proof}

  First consider the non-zero values of the above equation; use induction.
Let $P(n,k)$ be the statement, for fixed $k$, $\int_{-\pi}^{\pi} \frac{\cos (2nx)}{k+\sin^{2}x} dx =   \pi (-y_{n}  + z_{n} \sqrt{\frac{2}{T_{k}}})$.
We want: $P(n-2,k) \wedge P(n-1,k) \Rightarrow P(n,k)$. 
So by the assumption of $P(n,k)$ we have: 
 
\begin{equation*}
\int_{-\pi}^{\pi} \frac{(4k+2) \cos (2(n-1)x)}{k+\sin^{2}x} dx =  (4k+2) \pi (-y_{n-1}  + z_{n-1} \sqrt{\frac{2}{T_{k}}})
\end{equation*}

  So we must prove that $\int_{-\pi}^{\pi} \frac{(4k+2) \cos (2(n-1)x)  -  \cos (2(n-2)x) - \cos (2nx) }{k+\sin^{2}x} dx =  0$,
 which is elementary:
 
\begin{equation*}
\int_{-\pi}^{\pi} \frac{(4k+2) \cos (2(n-1)x)  -  \cos (2(n-2)x) - \cos (2nx) }{k+\sin^{2}x} dx =  \frac{4 \sin(2n \pi)}{n-1}
\end{equation*}

  By induction we must also prove that $\forall k \in \mathbb{R} (P(0,k) \wedge P(1,k) ) $.

  It is elementary to establish that:

\begin{equation*}\label{elementary1}
\int_{-\pi}^{\pi} \frac{1}{k + \sin ^{2}x } dx =  \frac{2 \sqrt{1 + \frac{1}{k}} \pi}{1 + k}
\end{equation*}

  So $P(0,k)$ holds for all $k$. As for $P(1,k)$, it is elementary to establish that:

\begin{equation*}
\int_{-\pi}^{\pi} \frac{\cos(2x)}{k + \sin ^{2}x } dx =  \frac{(-4 \sqrt{k} \sqrt{  1 + k} \pi + (1 + 2 k) (\pi + 2 i \ln (- \frac{i}{\sqrt{\frac{k}{1 + k}}}  )
 +     i \ln  (\frac{k}{1 + k})      ))}{\sqrt{k} \sqrt{1 + k}}
\end{equation*}

  So we must prove:

\begin{equation*}
\frac{(-4 \sqrt{k} \sqrt{  1 + k} \pi + (1 + 2 k) (\pi + 2 i \ln (- \frac{i}{\sqrt{\frac{k}{1 + k}}}  )
 +     i \ln  (\frac{k}{1 + k})      ))}{\sqrt{k} \sqrt{1 + k}} =  (-4 + (2 k + 1) \sqrt{\frac{4}{k(k+1)}} ) \pi
\end{equation*}

\begin{equation*}
\Leftrightarrow -4 \sqrt{k} \sqrt{  1 + k} \pi + (1 + 2 k) (\pi + 2 i \ln (- \frac{i}{\sqrt{\frac{k}{1 + k}}}  ) +     i \ln  (\frac{k}{1 + k})      ) 
= -4\sqrt{k}\sqrt{k+1}\pi + (4 k + 2)\pi  
\end{equation*}

\begin{equation*}
\Leftrightarrow  (1 + 2 k) (\pi + 2 i \ln (- \frac{i}{\sqrt{\frac{k}{1 + k}}}  ) +     i \ln  (\frac{k}{1 + k})      ) = (4 k + 2)\pi 
   \Leftrightarrow  2 i \ln (- \frac{i}{\sqrt{\frac{k}{1 + k}}}  ) +     i \ln  (\frac{k}{1 + k})      = \pi  
\end{equation*}

\begin{equation*}
\Leftrightarrow  \ln (- \frac{1+k}{ k}  ) +  \ln  (\frac{k}{1 + k})      = - \pi   i  \Leftrightarrow  \ln (- \frac{1+k}{ k}  ) +      \ln  (\frac{k}{1 + k})  
    = - \pi   i \Leftrightarrow  \ln (- 1  )     = - \pi   i 
\end{equation*}

  So the proof by induction is complete.

 Let $f(x,n,k)$ be an elementary periodic function of the form
 $f: \mathbb{Z}^{+} \times \mathbb{R}^{2} \to \mathbb{R}$ 
 given by the expression $\frac{\cos(( 2 n  + 1) x)} {k + \sin^{2} x  }$ for fixed integers $n$ and fixed $k$.

Clearly $f(n,k)$ is symmetric (about the $y$-axis).
 Interpret $f(x,n,k)$ as a shifted skew-symmetric function.
So we want that $f(x + \frac{\pi}{2},n,k) = -f(-x + \frac{\pi}{2},n,k)     $ . So we want:

\begin{equation*}
\frac{\cos(( 2 n  + 1) (x+\frac{\pi}{2}))} {k + \sin^{2} x  } = - \frac{\cos(( 2 n  + 1) (-x+\frac{\pi}{2}))} {k + \sin^{2} x  }
\end{equation*}

\begin{eqnarray*}
\lefteqn{\Leftrightarrow \cos(  ( 2 n  + 1)x  )      \cos (( 2 n  + 1)\frac{\pi}{2}   )   - \sin(  ( 2 n  + 1)x  )      \sin (( 2 n  + 1)\frac{\pi}{2}   )   =  }   \\
  &  &   -      (          \cos( ( 2 n  + 1)\frac{\pi}{2}  ) \cos  (   ( 2 n  + 1)x   )   +  \sin( ( 2 n  + 1)\frac{\pi}{2}  ) \sin  (   ( 2 n  + 1)x   ) )
\end{eqnarray*}

  Recall that $n$ is an integer, so $\cos (( 2 n  + 1)\frac{\pi}{2}   ) =0$. So the proof is complete.

\end{proof}

 Reinterpret $z_{n}$ as a multivariable function given by
 $z_{n,k} = (4k+2) z_{n-1,k} - z_{n-2,k}$; $z_{0,k}  = 1$, $z_{1,k}=2k+1$. Using OEIS one notices that:
$z_{n,1} =  \cosh(2 n \sinh^{-1} 1)$,
$z_{n,2} =  \cosh(2 n  \sinh^{-1}(\sqrt{2}))$,
$z_{n,3} =  \cosh(2 n  \sinh^{-1}(\sqrt{3}))$,
$z_{n,4} =  \cosh(2 n \sinh^{-1}(2))$, ...
  The pattern is clear.

 Reinterpret $y_{n,k}$ as a multivariable function given by
 $y_{n,k} = (4k+2) y_{n-1,k} - y_{n-2,k}$, $y_{0,k}  = 0$, $y_{1,k}=4$ . Using OEIS one notices that:

\begin{equation*}
y_{n,1} = \frac{1}{\sqrt{2}}    (  (3 + 2 \sqrt{2})^n - (3 - 2 \sqrt{2})^n) 
\end{equation*}

\begin{equation*}
y_{n,2}   =   \frac{1}{\sqrt{6}} ( (5 + 2 \sqrt{6})^n  - (5 - 2 \sqrt{6})^n)
\end{equation*}

 Given the above results one deduces that:

\begin{equation*}
y_{n,3}   =   \frac{1}{\sqrt{12}} ( (7 + 2 \sqrt{12})^n  - (7 - 2 \sqrt{12})^n)
\end{equation*}

  Also one may deduce that:

\begin{equation*}
y_{n,4}   =   \frac{1}{\sqrt{20}} ( (9 + 2 \sqrt{20})^n  - (9 - 2 \sqrt{20})^n)
\end{equation*}

  So we conjecture:

\begin{lemma}\label{lemma3}

\begin{equation*}
y_{n,k}   =   \frac{1}{\sqrt{k(k+1)}} ( (2k+1 + 2 \sqrt{k(k+1)})^n  - (2k+1 - 2 \sqrt{k(k+1)})^n)
\end{equation*}

\begin{equation*}
z_{n,k} =  \cosh(2 n  \sinh^{-1}(\sqrt{k}))
\end{equation*}

\end{lemma}

\begin{proof}

  We have 
 $y_{n,k} = (4k+2) y_{n-1,k} - y_{n-2,k}$, $y_{0,k}  = 0$, $y_{1,k}=4$, 
  and
 $z_{n,k} = (4k+2) z_{n-1,k} - z_{n-2,k}$; $z_{0,k}  = 1$, $z_{1,k}=2k+1$.

  Let $P(n,k)$ represent the statement `` $y_{n,k} = \frac{1}{\sqrt{k(k+1)}} ( (2k+1 + 2 \sqrt{k(k+1)})^n  - (2k+1 - 2 \sqrt{k(k+1)})^n)$'', for some
 fixed $k$. Verify that $P(0,k)$ and $P(1,k)$ hold:

\begin{equation*}
y_{0,k}   =   \frac{1}{\sqrt{k(k+1)}} ( (2k+1 + 2 \sqrt{k(k+1)})^0  - (2k+1 - 2 \sqrt{k(k+1)})^0)
\end{equation*}

\begin{equation*}
y_{1,k}   =   \frac{1}{\sqrt{k(k+1)}} ( (2k+1 + 2 \sqrt{k(k+1)})^1  - (2k+1 - 2 \sqrt{k(k+1)})^1)
\end{equation*}

  So by the assumption of $P(n,k)$ we want $P(n-2,k) \wedge P(n-1,k) \Rightarrow P(n,k) $, and have:

\begin{equation*}
(4k+2) y_{n-1,k}   = (4k+2)  \frac{1}{\sqrt{k(k+1)}} ( (2k+1 + 2 \sqrt{k(k+1)})^{n-1}  - (2k+1 - 2 \sqrt{k(k+1)})^{n-1})
\end{equation*}

\begin{equation*}
y_{n-2,k}   =   \frac{1}{\sqrt{k(k+1)}} ( (2k+1 + 2 \sqrt{k(k+1)})^{n-2}  - (2k+1 - 2 \sqrt{k(k+1)})^{n-2} )
\end{equation*}

 So it remains to prove that:

\begin{eqnarray*}
\lefteqn{\frac{1}{\sqrt{k(k+1)}} ( (2k+1 + 2 \sqrt{k(k+1)})^n  - (2k+1 - 2 \sqrt{k(k+1)})^n)   =  }    \\
  &  &   (4k+2)  \frac{1}{\sqrt{k(k+1)}} ( (2k+1 + 2 \sqrt{k(k+1)})^{n-1}  - (2k+1 - 2 \sqrt{k(k+1)})^{n-1})    \\
  &  &   -  \frac{1}{\sqrt{k(k+1)}} ( (2k+1 + 2 \sqrt{k(k+1)})^{n-2}  - (2k+1 - 2 \sqrt{k(k+1)})^{n-2} )
\end{eqnarray*}

\begin{eqnarray*}
\lefteqn{ \Leftrightarrow   (2k+1 + 2 \sqrt{k(k+1)})^n  - (2k+1 - 2 \sqrt{k(k+1)})^n   =  }    \\
  &  &   (4k+2)  ( (2k+1 + 2 \sqrt{k(k+1)})^{n-1}  - (2k+1 - 2 \sqrt{k(k+1)})^{n-1})    \\
  &  &   -  ( (2k+1 + 2 \sqrt{k(k+1)})^{n-2}  - (2k+1 - 2 \sqrt{k(k+1)})^{n-2} )
\end{eqnarray*}

 Simplify the right-hand side expression and the above equation holds.

  As for $z_{n,k}$, we have:
 $z_{n,k} = (4k+2) z_{n-1,k} - z_{n-2,k}$; $z_{0,k}  = 1$, $z_{1,k}=2k+1$, and we want
  $z_{0,k}  = 1$ and $z_{1,k}=2k+1$:

\begin{equation*}
z_{0,k} =  \cosh(2 \times 0  \sinh^{-1}(\sqrt{k}))
\end{equation*}

\begin{equation*}
z_{1,k} =  \cosh(2  \sinh^{-1}(\sqrt{k}))
\end{equation*}

  Expand and simplify, and the above holds.

\begin{equation*}
(4k+2)z_{n-1,k} = (4k+2) \cosh(2 (n-1)  \sinh^{-1}(\sqrt{k}))
\end{equation*}

\begin{equation*}
z_{n-2,k} =  \cosh(2 (n-2)  \sinh^{-1}(\sqrt{k}))
\end{equation*}

 So it remains to prove that:

\begin{eqnarray*}
\lefteqn{ \cosh(2 n  \sinh^{-1}(\sqrt{k}))=(4k+2) \cosh(2 (n-1)  \sinh^{-1}(\sqrt{k}))-\cosh(2 (n-2)  \sinh^{-1}(\sqrt{k}))  =  }    \\
  &  &   (4k+2) \cosh( 2n\sinh^{-1}(\sqrt{k})-2\sinh^{-1}(\sqrt{k})  )-\cosh(  2n\sinh^{-1}(\sqrt{k})-4\sinh^{-1}(\sqrt{k}) )
\end{eqnarray*}

  Expand and simplify, and the proof holds.

\end{proof}

\begin{theorem}\label{theorem2}

\begin{eqnarray*}
\lefteqn{ \int_{-\pi}^{\pi} \frac{\cos (nx)}{k+\sin^{2}x} dx = }    \\
  &  &   \frac{(-1)^{n}+1}{2}  \pi (-\frac{(2k+1 + 2 \sqrt{k(k+1)})^{\lfloor \frac{n}{2}  \rfloor} 
 - (2k+1 - 2 \sqrt{k(k+1)})^{\lfloor \frac{n}{2}  \rfloor}}{\sqrt{k(k+1)}}   \\
  &  &   +  \cosh(2 \lfloor \frac{n}{2}  \rfloor  \sinh^{-1}(\sqrt{k})) \frac{2}{\sqrt{k(k+1)}})
\end{eqnarray*}

\end{theorem}

\begin{proof}

  The above theorem follows from Lemma~\ref{lemma2} and Lemma~\ref{lemma3}.

\end{proof}

  Note that $k$ is not restricted to integers. Integrating the above with respect to $k$ yields strange results, often
related to Smarandache sequences.

 Consider the following function for even integers $n$:

\begin{equation*}
f(n) =  \frac{1}{48}    (-5+i^{n})(2+i^{n})n \int_{0}^{1} \int_{-\pi}^{\pi} \frac{\cos(nx)}{t+\sin^{2}x} dx dt
\end{equation*}

  Consider:

\begin{equation*}
f(10)  = -41 (-41 + 29 \sqrt{2}) \pi
\end{equation*}

\begin{equation*}
f(12)  = -140 (-140 + 99 \sqrt{2}) \pi
\end{equation*}

\begin{equation*}
f(14)  = -239 (-239 + 169 \sqrt{2}) \pi
\end{equation*}

\begin{equation*}
f(16)  = -816 (-816 + 577 \sqrt{2}) \pi
\end{equation*}

\begin{equation*}
f(18)  = -1393 (-1393 + 985 \sqrt{2}) \pi
\end{equation*}

  Remarkably the sequence ..., 41, 140, 239, 816, 1393, ... corresponds to the numerators of the lower principal and intermediate convergents to
 $\sqrt{2}$ \cite{A143608}, and the sequence ...,   29, 99, 169, 577, 985 , ... corresponds to the denominators of the lower principal and
 intermediate convergents to
 $\sqrt{2}$ \cite{A079496}.

\begin{conjecture}

 Let $A_{n}$ represent the numerator of the $n$th convergent to 
 $\sqrt{2}$ and let $P_{n}$ represent the $n$th Pell number, i.e. the denominator of the $n$th convergent to
 $\sqrt{2}$. 

 If $\frac{n}{2}$ is even then:

\begin{equation*}
\frac{1}{48}    (-5+i^{n})(2+i^{n})n \int_{0}^{1} \int_{-\pi}^{\pi} \frac{\cos(nx)}{t+\sin^{2}x} dx dt =   2 P_{n/2} (-2 P_{n/2} + A_{n/2} \sqrt{2}) \pi
\end{equation*}

 If $\frac{n}{2}$ is odd then:

\begin{equation*}
\frac{1}{48}    (-5+i^{n})(2+i^{n})n \int_{0}^{1} \int_{-\pi}^{\pi} \frac{\cos(nx)}{t+\sin^{2}x} dx dt =   A_{n/2} (-A_{n/2} + P_{n/2} \sqrt{2}) \pi
\end{equation*}

\end{conjecture}

  (Note that the double integral $\int_{-1}^{1} \int_{-\pi}^{\pi} \frac{\cos(nx)}{t+\sin^{2}x} dx dt$ displays similar 
behaviour.)

  Moreover,

\begin{conjecture}\label{conjecture2}

  For all integers $m$, and for all even integers $n$, the expression

\begin{equation*}
\int_{0}^{m} \int_{-\pi}^{\pi} \frac{\cos(nx)}{t+\sin^{2}x} dx dt
\end{equation*}

is equal to a number of the form $a(b+c\sqrt{d_{m}})\pi$, where $d_{m}$
is equal to the $m$th smallest $k$ such that $k^{2}(k+1)$ is a square, 
 and $a$, $b$, and $c$ are rational.

\end{conjecture}

  The aforementioned sequence (2, 6, 3, 5, 30, 42,...) is given in \cite{A083481}.
  To illustrate the above phenomenon, let $n=4$, and evaluate the above double integral for $m=1,2,3,...$, and 
note the appearance of the aforementioned integer sequence (2, 6, 3, 5, 30, 42, ...):

\begin{center}

$4 (-4 + 3 \sqrt{2}) \pi$

$4 (-12 + 5 \sqrt{6}) \pi$

$8 (-12 + 7 \sqrt{3}) \pi$

$8 (-20 + 9 \sqrt{5}) \pi$

$4 (-60 + 11 \sqrt{30}) \pi$

$4 (-84 + 13 \sqrt{42}) \pi$

$8 (-56 + 15 \sqrt{14}) \pi$

\end{center}

  Another example of the above double-integral-quadratic-convergent phenomenon:

\begin{conjecture}

   Let $C_{n}$ represent the numerator of the $n$th convergent to $\sqrt{3}$ and let $D_{n}$
 represent the denominator of the $n$th convergent to $\sqrt{3}$. If $n \equiv 2 \mod 4$, then:

\begin{equation*}
\frac{1}{4} n \int_{0}^{\frac{1}{2}} \int_{-\pi}^{\pi} \frac{\cos(nx)}{t+\sin^{2}x} dx dt  = C_{\frac{n}{2}} (-C_{\frac{n}{2}} + D_{\frac{n}{2}} \sqrt{3}) \pi
\end{equation*}

\end{conjecture}

  (A similar phenomenon holds for $n \equiv 0 \mod 4$.)

  One may easily conclude that the above phenomenon does not hold for all rational $m$; consider:

\begin{conjecture}

  For all even positive $n$,

\begin{equation*}
 n \int_{0}^{\frac{1}{3}} \int_{-\pi}^{\pi} \frac{\cos(nx)}{t+\sin^{2}x} dx dt  = \frac{ 4 \times 3^{\frac{n}{2}} -4 }{3^{\frac{n}{2}}} \pi
\end{equation*}

\end{conjecture}

  Further conjectures related to number theory follow:

\begin{conjecture}

  Consider the sequence of numbers given by:

\begin{equation*}
\int_{0}^{m} \int_{-\pi}^{\pi} \frac{\cos(2x)}{t+\sin^{2}x} dx dt
\end{equation*}

  Factoring the maximum integer from these expressions, the resulting coefficient of the square root is given by the sequence
 of the number of solutions to $x^{2} \equiv 0 \mod n$.

\end{conjecture}

  This sequence (1, 1, 2, 1, 1, 1, 2, 3, 1, 1, 2, 1, 1, 1, 4, 1, ...) is given by \cite{A000188}.
 Note the appearance of this sequence (and of \cite{A083481}),
 evaluating the above double integral for
 $m = 1,2,3,...$:

$4 (-1 + \sqrt{2}) \pi    $, 
$4 (-2 + \sqrt{6}) \pi     $, 
$4 (-3 + 2 \sqrt{3}) \pi$, 
$8 (-2 + \sqrt{5}) \pi$, 
$4 (-5 + \sqrt{30}) \pi$, 
$4 (-6 + \sqrt{42}) \pi$, 
$4 (-7 + 2 \sqrt{14}) \pi$, 
$8 (-4 + 3 \sqrt{2}) \pi$, 
$12 (-3 + \sqrt{10}) \pi$, 
$4 (-10 + \sqrt{110}) \pi$, 
$4 (-11 + 2 \sqrt{33}) \pi$, 
$8 (-6 + \sqrt{39}) \pi$, 
$4 (-13 + \sqrt{182}) \pi$, 
$4 (-14 + \sqrt{210}) \pi$, 
$4 (-15 + 4 \sqrt{15}) \pi$, 
$16 (-4 + \sqrt{17}) \pi$, ...

  Factoring the maximum integer from these expressions, note the sequence of (the absolute value of the)
 initial negative coefficients (1, 2, 3, 2, 5, 6, 7, 4, 3, 10, 11, 6, 13, 14, 15, 4, ...), which corresponds to the 
Smarandache sequence
 \cite{A019554}.
 We conjecture that:

\begin{conjecture}

  Consider the sequence of numbers given by:

\begin{equation*}
\int_{0}^{m} \int_{-\pi}^{\pi} \frac{\cos(2x)}{t+\sin^{2}x} dx dt
\end{equation*}

  Factoring the maximum integer from these expressions, the (absolute value of the) resulting negative
 coefficient is given by the sequence given by the smallest number whose square is divisible by $n$: 
(1, 2, 3, 2, 5, 6, 7, 4, 3, 10, 11, 6, 13, 14, 15, 4, 17, 6, ...).

\end{conjecture}

\begin{conjecture}

  Consider the sequence of numbers given by, for $m=1,2,3,...$:

\begin{equation*}
\int_{0}^{\frac{m+1}{m}} \int_{-\pi}^{\pi} \frac{\cos(2x)}{t+\sin^{2}x} dx dt
\end{equation*}

  Factoring the maximum integer from these numbers, the (absolutel value of the)
 resulting negative coefficient is given by the
 smallest number whose square is divisible by $n$. 

\end{conjecture}

  Again the Smarandache sequence \cite{A019554} arises.

  Consider the sequence of numbers given by, for $m=1,2,3,...$:

\begin{equation*}
\int_{0}^{\frac{m+1}{m}} \int_{-\pi}^{\pi} \frac{1}{t+\sin^{2}x} dx dt
\end{equation*}

  Unexpected patterns arise in the evaluation of the above double integral.
Let:

\begin{equation*}
\int_{0}^{\frac{m+1}{m}} \int_{-\pi}^{\pi} \frac{1}{t+\sin^{2}x} dx dt = -\pi (\ln (a_{m}^{2}) -2 \ln(b_{m} + c_{m} \sqrt{d_{m}}) )
\end{equation*}

  Using Sloane's online sequence recognition tool,
 we conjecture that:

\begin{conjecture}\label{conjecture8}

 $b_{m}$ corresponds to {\bf A165367}, the trisection $a(n) =$ {\bf A026741}  $(3*n+2)$, where 
{\bf A026741} is $n$ if $n$ odd, $n/2$ if $n$ even. 
  $a_{m}^{2}$ corresponds to {\bf A168077}, which is $a(2n)= ${\bf A129194}$(2n)/2$; $a(2n+1)=$ {\bf A129194}$(2n+1)$, 
where {\bf A129194}   is $n^{2} (\frac{3}{4}-(-1)^{\frac{n}{4}})$. 

\end{conjecture}

  We presently leave as open problems the evaluation of $c_{m}$ and $d_{m}$ (defined above).

Let:

\begin{equation*}
\int_{0}^{\frac{m+1}{m}} \int_{-\pi}^{\pi} \frac{\cos(4x)}{t+\sin^{2}x} dx dt = \pi a_{m} ( -b_{m} + c_{m} \sqrt{d_{m}}       )
\end{equation*}

  Using Sloane's online sequence recognition tool,
 we conjecture that:

\begin{conjecture}\label{conjecture9}

 $c_{m}$ corresponds to {\bf A165367}, the trisection $a(n) =$ {\bf A026741}  $(3*n+2)$, where 
{\bf A026741} is $n$ if $n$ odd, $n/2$ if $n$ even. 

\end{conjecture}

Let:

\begin{equation*}
\int_{0}^{m \times \sqrt{2}} \int_{-\pi}^{\pi} \frac{1}{t+\sin^{2}x} dx dt = 2 \pi \ln (1 + 2m \sqrt{2}  + a_{m} \sqrt{b_{m} (2m + \sqrt{2}) }    )
\end{equation*}

  Using Sloane's online sequence recognition tool,
 we conjecture that:

\begin{conjecture}\label{conjecture10}

$b_{m}$ corresponds to the Smarandache sequence {\bf A007913}, which is the squarefree part
 of $n: a(n) =$ smallest positive number $m$ such that
$ n/m$ is a square.  

\end{conjecture}

  Note that the Smarandache sequence corresponding to $b_{m}$ is also referred to as core($m$) \cite{A007913}. 
We thus have related Fourier-type integrals to squarefree numbers (through  \cite{A007913}), given that the above double 
integrals may be represented as follows, 
by Fubini's theorem:

\begin{equation*}
\int_{0}^{m} (\int_{-\pi}^{\pi} \frac{\cos(nx)}{t+\sin^{2}x} dx) dt
\end{equation*}

\begin{equation*}
=\int_{-\pi}^{\pi} \cos(n x) ( -\ln (\sin^{2} x ) + \ln(m + \sin^{2} x   )) dx
\end{equation*}

  Conjecture~\ref{conjecture10} may easily be extended to all non-perfect-square integers, with the Smarandache sequence 
{\bf A007913} \cite{A007913} again arising.

  We presently leave the above conjectures as open problems. A closed form evaluation of 
$\int_{q}^{m} \int_{-\pi}^{\pi} \frac{\cos(nx)}{t+\sin^{2}x} dx dt $ would be ideal, especially considering
 the multifarious connections to number theory.

\section{Other Results from Fourier Analysis}

  Consider elementary $2 \pi$-periodic functions of the form $f: \mathbb{Q}^{2} \times \mathbb{R} \to \mathbb{C}$ given by
 $f(x) = \ln(\sin (ax) +\cos (bx))$, where $a$ and $b$ are rational. Letting $a=1$ and $b=1$
the evaluation of the sine and cosine Fourier coefficients of $f(x)$ results in expressions involving 
$_{2} F _{1} (1,-\frac{n}{2},1-\frac{n}{2},i)$, which is not defined for positive even integers $n$. We evaluate the aforementioned coefficients as 
follows:

\begin{proposition}\label{firstpr}

\begin{eqnarray*}
\forall n \in \mathbb{Z}^{+}   (\lefteqn{\int_{-\pi}^{\pi} \cos(nx) \ln(\sin x + \cos x) dx =  }    \\  
&   &  (-1)^ {\lfloor {\frac{n + 3}{4}} \rfloor}       \frac{(-1)^{n + 1} + 1}{\sqrt{2}}    \frac{1}{2 \lfloor {\frac{n}{2}} \rfloor + 1}   i   \pi   
+ \lfloor {\frac{(n+3)\bmod{4}}{3}} \rfloor (-1)^{\lfloor \frac{n}{4} \rfloor}   \frac{2\pi}{n}  )
\end{eqnarray*}

\end{proposition}

\begin{proposition}\label{secondpr}

\begin{equation*}
\forall n \in \mathbb{Z}^{+}   (  \int_{-\pi}^{\pi} \sin(nx) \ln(\sin x + \cos x) dx 
 =   \frac   {     (-1)^ {\lfloor \frac{n + 5}  {4} \rfloor }  ((-1)^{n+3}+1) \pi i}{\sqrt{2} n} 
 - \frac{\lfloor  \frac{(n+5)\bmod{4} }{3}  \rfloor   (-1)^{\lfloor  \frac{n+2}{4}  \rfloor} \pi}{2 \lfloor \frac{n}{4} \rfloor +1})
\end{equation*}

\end{proposition}

   Note that the former series may be expressed using the Kronecker symbol and related modular forms \cite{A091337},
and that $\lfloor {\frac{(n+3)\bmod{4}}{3}} \rfloor (-1)^{\lfloor \frac{n}{4} \rfloor}$ corresponds to the inverse of the
8th cyclotomic polynomial \cite{A014017}.

  Given Proposition~\ref{firstpr} and Proposition~\ref{secondpr}, determine the Fourier series for $\ln(\sin x +\cos x)$:

\begin{eqnarray*}
\lefteqn{\ln(\sin x + \cos x) = \frac{i (\pi + i \ln2)}{2}  }      \\
&   &    + \sum_{n=1}^{\infty} \cos(nx) ((-1)^ {\lfloor {\frac{n + 3}{4}} \rfloor}  \frac{(-1)^{n + 1} + 1}{\sqrt{2}} 
   \frac{1}{2 \lfloor {\frac{n}{2}} \rfloor + 1}  i  + \lfloor {\frac{(n+3)\bmod{4}}{3}} \rfloor (-1)^{\lfloor \frac{n}{4} \rfloor}   \frac{2}{n}  )    \\
&   &    +\sin(nx) (\frac{(-1)^ {\lfloor \frac{n + 5}  {4} \rfloor }  ((-1)^{n+3}+1)  i}{\sqrt{2} n}            
- \frac{\lfloor  \frac{(n+5)\bmod{4} }{3}  \rfloor   (-1)^{\lfloor  \frac{n+2}{4}  \rfloor} }{2 \lfloor \frac{n}{4} \rfloor +1}))
\end{eqnarray*}
  
  Letting $x = \pi$, use the above technique to prove:

\begin{equation*}
-\frac{\pi} {\sqrt{2} } =  \sum_{n=1}^{\infty}   \frac{      (-1)^{\lfloor {  \frac{n+3}  {4}  } \rfloor}   ( (-1)^ {n+1} +1)    }       {2 \lfloor   \frac{n}{2} 
   \rfloor +1}
\end{equation*}

\begin{equation*}
-\frac{\ln 2}  {4}  = \sum_{n=1}^{\infty}  \frac{ \lfloor \frac{ (n+3)\bmod{4} }  {3} \rfloor    (-1)^{\lfloor \frac{n}{4} \rfloor}} {n}
\end{equation*}

  To conclude, in addition to presenting the above two series, we have rigorously proven the new evaluation of $\int_{-\pi}^{\pi} \frac{\cos (nx)}{k+\sin^{2}x} dx $, 
and have presented unusual connections between Fourier-type double integrals and Smarandache number theory.

\section{Appendix}

\emph{Mathematica} function for Theorem~\ref{theorem2}:

\vspace{.125in}

{\tt  Theorem2 =    \\
 Simplify[FunctionExpand[(-(((2\#1 + 1 + 2 Sqrt[\#1 (\#1 + 1)])\verb|^| \\
             Floor[\#2/2] - (2 \#1 + 1 - 2 Sqrt[\#1 (\#1 + 1)])\verb|^| \\
             Floor[\#2/2])/Sqrt[\#1 (\#1 + 1)]) + (Cosh[ \\
          2 Floor[\#2/2] ArcSinh[Sqrt[\#1]]]*(2/Sqrt[\#1 (\#1 + 1)])))* \\
     Pi*((((-1)\verb|^|\#2) + 1)/(2))]] \&}

\vspace{.125in}
  
  Integration results for Conjecture~\ref{conjecture8}:

\vspace{.125in}

\begin{center}

$ -\pi (\ln(4) -     2 \ln(7 + 3 \sqrt{5}))$,

$ -\pi (\ln(25) -     2 \ln(17 + 2 \sqrt{66}))$,

$ -\pi (\ln(9) -     2 \ln(10 + \sqrt{91}))$,

$ -\pi (\ln(49) -    2 \ln(23 + 4 \sqrt{30}))$,

$ -\pi (\ln(16) -    2 \ln(13 + 3 \sqrt{17}))$,

$ -\pi (\ln(81) -    2 \ln(29 + 2 \sqrt{190}))$,

$ -\pi (\ln(25) -    2 \ln(16 + \sqrt{231}))$,

$ -\pi (\ln(121) -     2 \ln(35 + 4 \sqrt{69}))$,

$ -\pi (\ln(36) -     2 \ln(19 + 5 \sqrt{13}))$,

$ -\pi (\ln(169) -     2 \ln(41 + 6 \sqrt{42}))$,

$ -\pi (\ln(49) -     2 \ln(22 + \sqrt{435}))$,

$ -\pi (\ln(225) -     2 \ln(47 + 8 \sqrt{31}))$,

$ -\pi (\ln(64) -     2 \ln(25 + \sqrt{561}))$,

$ -\pi (\ln(289) -     2 \ln(53 + 6 \sqrt{70}))$,

$ -\pi (\ln(81) -     2 \ln(28 + \sqrt{703}))$,

$ -\pi (\ln(361) -     2 \ln(59 + 4 \sqrt{195}))$,

$ -\pi (\ln(100) -     2 \ln(31 + \sqrt{861})) $, ...

\end{center}

\vspace{.125in}

  As indicated above, the sequence 4, 25, 9, 49, 16, 81, 25, 121, 36, 169, 49, 225, 64, 289, 81, 361, 100, ... 
corresponds to {\bf A168077} \cite{A168077}, and the sequence 
7, 17, 10, 23, 13, 29, 16, 35, 19, 41, 22, 47, 25, 53, 28, 59, 31, ... corresponds to {\bf A165367} \cite{A165367}.

\vspace{.125in}
  
  Integration results for Conjecture~\ref{conjecture9}:

\vspace{.125in}

\begin{center}

$4 (-12 + 5 \sqrt{6}) \pi$, 

$2 (-15 + 4 \sqrt{15}) \pi $, 

$\frac{8}{9} (-28 + 11 \sqrt{7}) \pi$, 

$\frac{3}{2} (-15 + 7 \sqrt{5}) \pi $, 

$\frac{4}{25} (-132 + 17 \sqrt{66}) \pi$, 

$\frac{2}{9} (-91 + 10 \sqrt{91}) \pi $, 

$\frac{8}{49} (-120 + 23 \sqrt{30}) \pi$, 

$\frac{3}{8} (-51 + 13 \sqrt{17}) \pi $, 

$\frac{4}{81} (-380 + 29 \sqrt{190}) \pi$, 

$\frac{2}{25} (-231 + 16 \sqrt{231}) \pi $, 

$\frac{8}{121} (-276 + 35 \sqrt{69}) \pi$, 

$\frac{5}{18} (-65 + 19 \sqrt{13}) \pi $, 

$\frac{12}{169} (-252 + 41 \sqrt{42}) \pi$, 

$\frac{2}{49} (-435 + 22 \sqrt{435}) \pi$, 

$\frac{16}{225} (-248 + 47 \sqrt{31}) \pi $, 

$\frac{1}{32} (-561 + 25 \sqrt{561}) \pi$, 

$\frac{12}{289} (-420 + 53 \sqrt{70}) \pi$, 

$\frac{2}{81} (-703 + 28 \sqrt{703}) \pi $, 

$\frac{8}{361} (-780 + 59 \sqrt{195}) \pi $, 

$\frac{1}{50} (-861 + 31 \sqrt{861}) \pi$, ...

\end{center}

\vspace{.125in}

As indicated above the sequence 
5, 4, 11, 7, 17, 10, 23, 13, 29, 16, 35, 19, 41, 22, 47, 25, 53, 28, 59, 31
 corresponds to {\bf A165367} \cite{A165367}.

\vspace{.125in}

  Integration results for Conjecture~\ref{conjecture10}:

\vspace{.125in}

\begin{center}

$2 \pi \ln(1 + 2 \sqrt{2} + 2 \sqrt{2 + \sqrt{2}})$, 

$2 \pi \ln(1 + 4 \sqrt{2} + 2 \sqrt{2 (4 + \sqrt{2})})$, 

$2 \pi \ln(1 + 6 \sqrt{2} + 2 \sqrt{3 (6 + \sqrt{2})})$, 

$2 \pi \ln(1 + 8 \sqrt{2} + 4 \sqrt{8 + \sqrt{2}})$, 

$2 \pi \ln(1 + 10 \sqrt{2} + 2 \sqrt{5 (10 + \sqrt{2})})$, 

$2 \pi \ln(1 + 12 \sqrt{2} + 2 \sqrt{6 (12 + \sqrt{2})})$, 

$2 \pi \ln(1 + 14 \sqrt{2} + 2 \sqrt{7 (14 + \sqrt{2})})$, 

$2 \pi \ln(1 + 16 \sqrt{2} + 4 \sqrt{2 (16 + \sqrt{2})}) $, 

$2 \pi \ln[1 + 18 \sqrt{2} + 6 \sqrt{18 + \sqrt{2})}) $, 

$2 \pi \ln[1 + 20 \sqrt{2} + 2 \sqrt{10 (20 + \sqrt{2})}) $, 

$2 \pi \ln[1 + 22 \sqrt{2} + 2 \sqrt{11 (22 + \sqrt{2})}) $, 

$2 \pi \ln[1 + 24 \sqrt{2} + 4 \sqrt{3 (24 + \sqrt{2})}) $, 

$2 \pi \ln[1 + 26 \sqrt{2} + 2 \sqrt{13 (26 + \sqrt{2})}) $, 

$2 \pi \ln[1 + 28 \sqrt{2} + 2 \sqrt{14 (28 + \sqrt{2})}) $, 

$2 \pi \ln[1 + 30 \sqrt{2} + 2 \sqrt{15 (30 + \sqrt{2})}) $, 

$2 \pi \ln[1 + 32 \sqrt{2} + 8 \sqrt{32 + \sqrt{2})}) $, 

$2 \pi \ln[1 + 34 \sqrt{2} + 2 \sqrt{17 (34 + \sqrt{2})}) $, 

$2 \pi \ln[1 + 36 \sqrt{2} + 6 \sqrt{2 (36 + \sqrt{2})}) $, 

$2 \pi \ln[1 + 38 \sqrt{2} + 2 \sqrt{19 (38 + \sqrt{2})}) $, 

$2 \pi \ln[1 + 40 \sqrt{2} + 4 \sqrt{5 (40 + \sqrt{2})}) $, ...

\end{center}

\vspace{.125in}

  As indicated above, the sequence
1, 2, 3, 1, 5, 6, 7, 2, 1, 10, 11, 3, 13, 14, 15, 1, 17, 2, 19, 5, ...     
 corresponds to the Smarandache sequence {\bf A007913} \cite{A007913}.

\end{document}